\documentclass[10pt]{article}
\usepackage{amsfonts}
\usepackage{amsmath}
\usepackage{color}

\newtheorem{theorem}{Theorem}

\newtheorem{lemma}[theorem]{Lemma}

\newenvironment{proof}[1][Proof]{\noindent\textbf{#1.} }{\ \rule{0.5em}{0.5em}}
\input{tcilatex}
\begin{document}

\title{On homoclinic solutions for a second order difference equation with $%
p-$Laplacian}
\author{Robert Stegli\'{n}ski}
\maketitle

\begin{abstract}
In this paper, we obtain conditions under which the difference equation 
\begin{equation*}
-\Delta \left( a(k)\phi _{p}(\Delta u(k-1))\right) +b(k)\phi
_{p}(u(k))=\lambda f(k,u(k)),\quad k\in 
\mathbb{Z}
,
\end{equation*}%
has infinitely many homoclinic solutions. A variant of the fountain theorem
is utilized in the proof of our theorem. It improves the results in [L.
Kong, Homoclinic solutions for a second order difference equation with $p-$%
Laplacian, \textit{Appl. Math. Comput}., \textbf{247} (2014), 1113--1121],
where the set of conditions imposed on nonlinearity is inconsistent.
\end{abstract}

\textbf{Math Subject Classifications}: 39A10, 47J30, 35B38


\section{Introduction}

\bigskip In the present paper we deal with the following nonlinear
second-order difference equation: 
\begin{equation}
\left\{ 
\begin{array}{ll}
-\Delta \left( a(k)\phi _{p}(\Delta u(k-1))\right) +b(k)\phi
_{p}(u(k))=\lambda f(k,u(k)) & \mbox{for
all $k\in\mathbb{Z}$} \\ 
u(k)\rightarrow 0 & \mbox{as $|k|\to \infty$}.%
\end{array}%
\right.  \label{eq}
\end{equation}%
Here $p>1$ is a real number, $\lambda $ is a positive real parametr, $\phi
_{p}(t)=|t|^{p-2}t$ for all $t\in {\mathbb{R}}$, $a,b:{\mathbb{Z}}%
\rightarrow \mathbb{(}0,+\infty )$, while $f:{\mathbb{Z}}\times {\mathbb{R}}%
\rightarrow {\mathbb{R}}$ is a continuous function. Moreover, the forward
difference operator is defined as $\Delta u(k-1)=u(k)-u(k-1)$. We say that a
solution $u=\{u(k)\}$ of (\ref{eq})\ is homoclinic if $\lim_{\left\vert
k\right\vert \rightarrow \infty }u(k)=0.$

\bigskip In this paper, similary to \cite{K}, our goal is to apply the
variational method and a variant of the fountain theorem to find a sequence
of homoclinic solutions for the problem (\ref{eq}). Our theorem improves the
results in \cite{K}, where the set of conditions imposed on nonlinearity is
inconsistent. We not only show that one of the assumptions is in fact
superfluous, but also that others can be relaxed. The problem (\ref{eq}) has
been studied recently in several papers. Infinitely many solutions were
obtained in \cite{SM} by employing Nehari manifold methods, in \cite{St} by
use of the Ricceri's theorem (see \cite{BMB}, \cite{R}), and in \cite{St1}
directly applying the variational method.

\bigskip We assume that potential $b(k)$ and the nonlinearity $f(k,t)$
satisfies the following conditions:

\begin{itemize}
\item[$(B)$] $b(k)\geq b_{0}>0$ for all $k\in 
\mathbb{Z}
$, $b(k)\rightarrow +\infty $ as $\left\vert k\right\vert \rightarrow
+\infty ;$

\item[$(H_{1})$] $\displaystyle f(k,-t)=-f(k,t)$ for all $k\in 
\mathbb{Z}
$ and $t\in 
\mathbb{R}
;$

\item[$(H_{2})$] $\displaystyle$there exist $d>0$ and $q>p$ such that $%
~\left\vert F(k,t)\right\vert \leq d\left( \left\vert t\right\vert
^{p}+\left\vert t^{q}\right\vert \right) $ for all $k\in 
\mathbb{Z}
$ and $t\in 
\mathbb{R}
;$

\item[$(H_{3})$] $\displaystyle\lim_{t\rightarrow 0}\frac{f(k,t)}{\left\vert
t\right\vert ^{p-1}}=0$ uniformly for all $k\in {\mathbb{Z}}$;

\item[$(H_{4})$] $\displaystyle\lim\limits_{\left\vert t\right\vert
\rightarrow +\infty }\frac{f(k,t)t}{\left\vert t\right\vert ^{p}}=+\infty $
for all $k\in 
\mathbb{Z}
;$

\item[$(H_{5})$] $\displaystyle$there exists $\sigma \geq 1$ such that $%
\sigma \mathcal{F}(k,t)\geq \mathcal{F}(k,st)$ for $k\in 
\mathbb{Z}
,t\in 
\mathbb{R}
,$ and $s\in \lbrack 0,1],$
\end{itemize}

where $F(k,t)$ is the primitive function of $f(k,t)$, that is $%
F(k,t)=\int_{0}^{t}f(k,s)ds$ for every $t\in 
\mathbb{R}
$ and $k\in 
\mathbb{Z}
,$ and $\mathcal{F}(k,t)=f(k,t)t-pF(k,t)$.

\bigskip Kong \cite{K} gave conditions for existence of a sequence of
solutions of the problem (\ref{eq}). In additions to hypotheses $%
(B),(H_{1}),(H_{5})$, he offered also the following conditions:

\begin{itemize}
\item[$(H_{2}^{\prime })$] $\displaystyle$there exist $d>0$ and $q>p$ such
that $~\left\vert F(k,t)\right\vert \leq d\left\vert t\right\vert ^{q}$ for
all $k\in 
\mathbb{Z}
$ and $t\in 
\mathbb{R}
;$

\item[$(H_{3}^{\prime })$] $\displaystyle\sup_{\left\vert t\right\vert \leq
T}\left\vert F(\cdot .t)\right\vert \in l_{1}$ for all $T>0$;

\item[$(H_{4}^{\prime })$] $\displaystyle\lim\limits_{\left\vert
t\right\vert \rightarrow +\infty }\frac{f(k,t)t}{\left\vert t\right\vert ^{p}%
}=+\infty $ uniformly for all $k\in 
\mathbb{Z}
.$
\end{itemize}

\bigskip Obviously, $(H_{4}^{\prime })$ is stronger than $(H_{4})$, and $%
(H_{2}^{\prime })$\ is stronger than both $(H_{2})$ and $(H_{3}).$\ In \cite%
{K}, as an example of function, which satisfied conditions $%
(H_{1}),(H_{2}^{\prime }),(H_{3}^{\prime }),(H_{4}^{\prime }),(H_{5})$ is
given the function 
\begin{equation}
f(k,t)=\frac{1}{k^{\mu }}\left\vert t\right\vert ^{p-2}t\ln \left(
1+\left\vert t\right\vert ^{\nu }\right) ,\ \ \ \ \ (k,t)\in 
\mathbb{Z}
\times 
\mathbb{R}
\label{function}
\end{equation}%
with $\mu >1$ and $\nu \geq 1$. But this does not satisfy the condition $%
(H_{4}^{\prime })$. Moreover, the conditions $(H_{3}^{\prime })$ and $%
(H_{4}^{\prime })$ are contradictory. Indeed, since $p>1$ the hypothesis $%
(H_{4}^{\prime })$ does give us $T_{1}>0$ such that $\left\vert
f(k,t)\right\vert \geq 1$ for all $\left\vert t\right\vert \geq T_{1}$ and $%
k\in 
\mathbb{Z}
$. Put $\alpha _{k}=F(k,T_{1})$ for all$\ k\in 
\mathbb{Z}
$. Then $\{\alpha _{k}\}\in l_{1},$ by $(H_{3}^{\prime })$. As $f$ is
continuous we have for $T>T_{1}$ and $k\in 
\mathbb{Z}
$%
\begin{eqnarray*}
\left\vert F(k,T)\right\vert &=&\left\vert \int_{0}^{T}f(k,t)dt\right\vert
=\left\vert \int_{0}^{T_{1}}f(k,t)dt+\int_{T_{1}}^{T}f(k,t)dt\right\vert
=\left\vert \alpha _{k}+\int_{T_{1}}^{T}f(k,t)dt\right\vert \\
&\geq &\left\vert \int_{T_{1}}^{T}f(k,t)dt\right\vert -\left\vert \alpha
_{k}\right\vert =\int_{T_{1}}^{T}\left\vert f(k,t)\right\vert dt-\left\vert
\alpha _{k}\right\vert \geq (T-T_{1})-\left\vert \alpha _{k}\right\vert ,
\end{eqnarray*}%
and so $\left\vert F(\cdot ,T)\right\vert \notin l_{1}$, contrary to $%
(H_{3}^{\prime })$.

\ It is easy to verify that the function (\ref{function}) does satisfy
conditions $(H_{1})-(H_{5}).$\ Note also that it does not satisfy the
standard Ambrosetti-Rabinowitz condition.

\section{\protect\bigskip Preliminaries}

We repeat the relevant for us material from \cite{K}. We begin by defining
some Banach spaces. For all $1\leq p<+\infty $, we denote $\ell ^{p}$ the
set of all functions $u:{\mathbb{Z}}\rightarrow {\mathbb{R}}$ such that 
\begin{equation*}
\Vert u\Vert _{p}^{p}=\sum_{k\in {\mathbb{Z}}}|u(k)|^{p}<+\infty .
\end{equation*}%
Moreover, we denote $\ell ^{\infty }$ the set of all functions $u:{\mathbb{Z}%
}\rightarrow {\mathbb{R}}$ such that 
\begin{equation*}
\Vert u\Vert _{\infty }=\sup_{k\in {\mathbb{Z}}}|u(k)|<+\infty
\end{equation*}

\bigskip We set 
\begin{equation}
X=\left\{ u:{\mathbb{Z}}\rightarrow {\mathbb{R}}\ :\ \ \sum_{k\in {\mathbb{Z}%
}}\left[ a(k)\left\vert \Delta u(k-1)\right\vert ^{p}+b(k)|u(k)|^{p}\right]
<\infty \right\}  \label{X}
\end{equation}%
and 
\begin{equation*}
\Vert u\Vert =\left( \sum_{k\in {\mathbb{Z}}}\left[ a(k)\left\vert \Delta
u(k-1)\right\vert ^{p}+b(k)|u(k)|^{p}\right] \right) ^{\frac{1}{p}}.
\end{equation*}%
Clearly we have 
\begin{equation}
\Vert u\Vert _{\infty }\leq \Vert u\Vert _{p}\leq b_{0}^{-\frac{1}{p}}\Vert
u\Vert \ \mbox{for all $u\in X$.}  \label{a}
\end{equation}

\bigskip Moreover, $(X,\Vert \cdot \Vert )$ is a reflexive and separable
Banach space and the embedding $X\hookrightarrow l^{p}$ is compact (see
Lemma 2.2 in \cite{K}).

\begin{lemma}
\label{ogon}\bigskip If $S$ is a compact subset of $l^{p}$, then, for every $%
\delta >0$, there exists $h>0$ such that 
\begin{equation*}
\left\vert u(k)\right\vert \leq \left( \sum_{\left\vert k\right\vert
>h}\left\vert u(k)\right\vert ^{p}\right) ^{\frac{1}{p}}<\delta
\end{equation*}%
for all $u\in S$ and $\left\vert k\right\vert >h.$
\end{lemma}

\bigskip This is Lemma 3.3 in \cite{K}.

Let 
\begin{equation*}
\Phi (u):=\frac{1}{p}\sum_{k\in 
\mathbb{Z}
}\left[ a(k)\left\vert \Delta u(k-1)\right\vert ^{p}+b(k)\left\vert
u(k)\right\vert ^{p}\right] \ \ \ \text{for all \ \ \ }u\in X
\end{equation*}%
and 
\begin{equation*}
\Psi (u):=\sum_{k\in 
\mathbb{Z}
}F(k,u(k))\text{ \ \ for all \ \ }u\in l^{p}
\end{equation*}%
where $F(k,s)=\int_{0}^{s}f(k,t)dt$ for $s\in 
\mathbb{R}
$ and $k\in 
\mathbb{Z}
$.\ Let $J:X\rightarrow \mathbb{R}$ be the functional associated to problem (%
\ref{eq}) defined by 
\begin{equation*}
J_{\lambda }(u)=\Phi (u)-\lambda \Psi (u).
\end{equation*}

\begin{lemma}
\label{propIT}\bigskip Assume that $(B)$ and $(H_{3})$ are satisfied. Then

\begin{itemize}
\item[$(a)$] $\Phi \in C^{1}(X);$

\item[$(b)$] $\Psi \in C^{1}(l^{p})$ $\ $and \ $\Psi \in C^{1}(X)$;

\item[$(c)$] $J_{\lambda }\in C^{1}(X)$ and every critical point $u\in X$ of 
$J_{\lambda }$ is a homoclinic solution of problem (\ref{eq}).
\end{itemize}
\end{lemma}

This version of the lemma can be proved essentially by the same way as
Propositions 5,6 and 7 in \cite{IT}, where $a(k)\equiv 1$ on $%
\mathbb{Z}
$ and the norm on $X$ is slightly different. See also Lemma 2.3 in \cite{K}.

\bigskip

\section{\protect\bigskip \protect\bigskip Main results}

Now we are ready to state our result.

\begin{theorem}
\label{main}Suppose that the conditions $(B),(H_{1})-(H_{5})$\ hold. Then,
for any $\lambda >0,$ the problem (\ref{eq}) has a sequence $\{u_{n}(k)\}$
of solutions such that $J_{\lambda }(u_{n})\rightarrow \infty $ as $%
n\rightarrow \infty $.
\end{theorem}

\bigskip

Our main tool is the following version of the fountain theorem with Cerami's
condition (see \cite{L}). We say that $I$, a $C^{1}$-functional defined on a
Banach space $X$, satisfies the Cerami condition if any sequence $%
\{u_{n}\}\subset X$ such that $\{I(u_{n})\}$ is bounded and $(1+\left\Vert
u_{n}\right\Vert )\left\Vert I^{\prime }(u_{n})\right\Vert \rightarrow 0$
has a convergent subsequence; such a sequence is then called a Cerami
sequence. Now, let $X$ be a reflexive and separable Banach space. It is well
known that there exists $e_{i}\in X$ and $e_{i}^{\ast }\in X^{\ast }$ such
that 
\begin{equation*}
X=\overline{\text{span}\{e_{i}:i\in 
\mathbb{N}
\}},\ \ \ \ X^{\ast }=\overline{\text{span}\{e_{i}^{\ast }:i\in 
\mathbb{N}
\}}^{w^{\ast }}
\end{equation*}%
and%
\begin{equation*}
\left\langle e_{i}^{\ast },e_{j}\right\rangle =\delta _{ij}\text{, where }%
\delta _{ij}=1\text{ for }i=j\text{ and }\delta _{ij}=0\text{ for }i\neq j.
\end{equation*}%
Put%
\begin{equation}
X_{i}=\text{span}\{e_{i}\},\ \ \ \ Y_{n}=\oplus _{i=1}^{n}X_{i}\ \ \ \text{%
and}\ \ \ Z_{n}=\overline{\oplus _{i=n}^{\infty }X_{i}}.  \label{YZ}
\end{equation}

\begin{theorem}
\label{fountain}\bigskip Assume that $I\in C^{1}(X,%
\mathbb{R}
)$ satisfies the Cerami condition and $I(-u)=I(u)$. If for almost every $%
n\in 
\mathbb{N}
$, there exist $\rho _{n}>r_{n}>0$ such that

\begin{itemize}
\item[$(i)$] $a_{n}=\inf_{u\in Z_{n},\left\Vert u\right\Vert
=r_{n}}I(u)\rightarrow +\infty $ as $n\rightarrow \infty ;$

\item[$(ii)$] $b_{n}=\max_{u\in Y_{n},\left\Vert u\right\Vert =\rho
_{n}}I(u)\leq 0$,

then $I$ has a sequence of critical points $\{u_{n}\}$ such that $%
I(u_{n})\rightarrow +\infty $.
\end{itemize}
\end{theorem}

\bigskip In the remainder of this paper, let $X$ be defined by (\ref{X}),
and $Y_{n}$ and $Z_{n}$ be given in (\ref{YZ}). To prove Theorem 1, we will
also need the following.

\begin{lemma}
\label{beta}\bigskip Let $q\geq p$. For $n\in 
\mathbb{N}
$, define 
\begin{equation*}
\beta _{q,n}=\sup_{u\in Z_{n},\left\Vert u\right\Vert =1}\left\Vert
u\right\Vert _{q}.
\end{equation*}%
Then, $\lim_{n\rightarrow \infty }\beta _{q,n}=0$.
\end{lemma}

\bigskip For $q>p$ this is Lemma 3.2 in \cite{K}. As proof shows, the
instance $q=p$ is also true.

\begin{lemma}
\label{pos}\bigskip Suppose that $(H_{5})$ holds, then 
\begin{equation*}
F(k,t)\geq 0\ \ \ \text{and}\ \ \ f(k,t)t\geq 0
\end{equation*}%
for all $(k,t)\in 
\mathbb{Z}
\times 
\mathbb{R}
$.
\end{lemma}

\begin{proof}
\bigskip By $(H_{5})$, 
\begin{equation*}
f(k,t)t-pF(k,t)\geq 0
\end{equation*}%
for all $(k,t)\in 
\mathbb{Z}
\times 
\mathbb{R}
$. From this, for $t>0$ and $k\in 
\mathbb{Z}
$, we have 
\begin{equation*}
\frac{\partial }{\partial t}\left( \frac{F(k,t)}{t^{p}}\right) =\frac{%
t^{p}f(k,t)-pt^{p-1}F(k,t)}{t^{2p}}\geq 0.
\end{equation*}%
Since $F(k,0)=0$, we obtain $F(k,t)\geq 0$ and $f(k,t)t\geq 0$ for all $k\in 
\mathbb{Z}
$ and $t\geq 0$. Arguing similarly for the case $t\leq 0$, we complete the
proof.
\end{proof}

\begin{lemma}
\label{cerami}\bigskip Assume that $(B)$ and $(H_{3}),(H_{4}),(H_{5})$ hold.
Then, for any $\lambda >0$, $J_{\lambda }$ satisfies Cerami's condition.
\end{lemma}

\begin{proof}
Let $\lambda >0$ be fixed. Let $\{u_{n}\}$ be a Cerami sequence of $%
J_{\lambda }$. Firstly, we assume that $\{u_{n}\}$ is bounded. Up to
considering a subsequence, we may assume that for some $c\in 
\mathbb{R}
$, $J_{\lambda }(u_{n})\rightarrow c$ and $J_{\lambda }^{\prime
}(u_{n})\rightarrow 0$. Then the proof proceeds along the same lines as the
first part of the proof of Lemma 3.4 in \cite{K}, giving the desired
conclusion. This proof uses assumptions $(B)$ and $(H_{3})$.

Now, let we suppose that $\{u_{n}\}$ is unbounded. Then, up to a
subsequence, we may assume that for some $c\in 
\mathbb{R}
$,%
\begin{equation*}
J_{\lambda }(u_{n})\rightarrow c,\ \ \ \left\Vert u_{n}\right\Vert
\rightarrow \infty ,\ \ \ \left\Vert u_{n}\right\Vert \left\Vert J_{\lambda
}^{\prime }(u_{n})\right\Vert \rightarrow 0.
\end{equation*}

\bigskip And again, proceeding as in the second part of the proof of Lemma
3.4 in \cite{K}, we obtain a contradiction. We must only use Lemma \ref{pos}
and assumption $(H_{4})$ instead of using condition $(H_{4}^{\prime })$.
\end{proof}

\bigskip

\begin{proof}[Proof of Theorem 1]
\bigskip Let $\lambda >0$ be fixed. By $(H_{1})$ and Lemma \ref{cerami}, $%
J_{\lambda }$ is even and satisfies Cerami's condition. In the following, we
show that, for almost every $n\in 
\mathbb{N}
$, there exist $\rho _{n}>r_{n}>0$ such that conditions $(i)$ and $(ii)$ of
Theorem \ref{fountain}\ with $I=J_{\lambda }$ are satisfied.

By Lemma \ref{ogon}, there exists $n_{0}\in 
\mathbb{N}
$ such that $\beta _{p,n}<\left( \frac{1}{2p\lambda d}\right) ^{\frac{1}{p}}$
for all $n\geq n_{0}$.\ Let $r_{n}=\left( \frac{\frac{1}{2p}-\lambda d\beta
_{p,n}^{p}}{\lambda d\beta _{q,n}^{q}}\right) ^{\frac{1}{q-p}}$ for all $%
n\geq n_{0}$. Here $d>0$ and $q>p$ are given in $(H_{2}).$ Then, one has $%
r_{n}>0$ and $\lim_{n\rightarrow \infty }r_{n}=\infty $. For $u\in Z_{n}$ we
have $\left\Vert u\right\Vert _{p}\leq \beta _{p,n}\left\Vert u\right\Vert $
and $\left\Vert u\right\Vert _{q}\leq \beta _{q,n}\left\Vert u\right\Vert $,
by Lemma \ref{beta}. Let $u\in Z_{n}$\ with $\left\Vert u\right\Vert =r_{n}$%
. Then 
\begin{eqnarray*}
J_{\lambda }(u) &=&\frac{1}{p}\sum_{k\in 
\mathbb{Z}
}\left[ a(k)\left\vert \Delta u(k-1)\right\vert ^{p}+b(k)\left\vert
u(k)\right\vert ^{p}\right] -\lambda \sum_{k\in 
\mathbb{Z}
}F(k,u(k))\geq \frac{1}{p}\left\Vert u\right\Vert ^{p}-\lambda d\left(
\left\Vert u\right\Vert _{p}^{p}+\left\Vert u\right\Vert _{q}^{q}\right) \\
&\geq &\frac{1}{p}\left\Vert u\right\Vert ^{p}-\lambda d\left( \beta
_{p,n}^{p}\left\Vert u\right\Vert ^{p}+\beta _{q,n}^{q}\left\Vert
u\right\Vert ^{q}\right) =\frac{1}{2p}\left( \frac{\frac{1}{2p}-\lambda
d\beta _{p,n}^{p}}{\lambda d\beta _{q,n}^{q}}\right) ^{\frac{p}{q-p}}=\frac{1%
}{2p}r_{n}^{p}.
\end{eqnarray*}%
Therefore%
\begin{equation*}
a_{n}=\inf_{u\in Z_{n},\left\Vert u\right\Vert =r_{n}}J_{\lambda
}(u)\rightarrow +\infty \ \ \ \ \ \ \text{as \ \ \ }n\rightarrow \infty
\end{equation*}%
and condition $(i)$ holds.

Since $\dim Y_{n}<\infty $, there exists $C_{n}>0$ such that 
\begin{equation}
\frac{1}{p}\left\Vert u\right\Vert ^{p}\leq \lambda C_{n}\left\Vert
u\right\Vert _{\infty }^{p}\text{ \ \ \ \ \ \ \ for all }u\in Y_{n}.
\label{nor}
\end{equation}%
As $S_{n}=\{u\in Y_{n}:\left\Vert u\right\Vert =1\}$ is compact in $l^{p}$,
there exists $h_{n}>0$ such that $\left\vert u(k)\right\vert <1$ for all $%
u\in S_{n}$ and $\left\vert k\right\vert >h_{n}$, by Lemma \ref{ogon}.
Consequently, for every $u\in Y_{n}$, there exists $k_{0}\in 
\mathbb{Z}
$ with $\left\vert k_{0}\right\vert \leq h_{n}$ such that $\left\Vert
u\right\Vert _{\infty }=\left\vert u(k_{0})\right\vert $. By $(H_{4})$,
there exists $T>0$ such that%
\begin{equation}
F(k,t)\geq 2C_{n}\left\vert t\right\vert ^{p}\text{ \ \ \ \ \ \ \ \ for }%
(k,t)\in 
\mathbb{Z}
\times 
\mathbb{R}
\text{ with }\left\vert k\right\vert \leq h_{n}\text{ and }\left\vert
t\right\vert >T.  \label{osz}
\end{equation}%
Choose $\rho _{n}>\max \{\left( \lambda pC_{n}\right) ^{1/p}T,r_{n}\}$ for
all $n\geq n_{0}$. Then $\rho _{n}>r_{n}>0$. Taking $u\in Y_{n}$ with $%
\left\Vert u\right\Vert =\rho _{n}$ we have $\left\Vert u\right\Vert
_{\infty }>T$ and $\left\Vert u\right\Vert _{\infty }=\left\vert
u(k_{0})\right\vert $\ for some $k_{0}\in 
\mathbb{Z}
$ with $\left\vert k_{0}\right\vert \leq h_{n}$. Thus, from (\ref{nor}),(\ref%
{osz}) and Lemma \ref{pos}, it follows that%
\begin{eqnarray*}
J_{\lambda }(u) &=&\frac{1}{p}\left\Vert u\right\Vert ^{p}-\lambda
\sum_{k\in 
\mathbb{Z}
}F(k,u(k))\leq \frac{1}{p}\left\Vert u\right\Vert ^{p}-\lambda
F(k_{0},u(k_{0}))\leq \frac{1}{p}\left\Vert u\right\Vert ^{p}-\lambda
2C_{n}\left\vert u(k_{0})\right\vert ^{p} \\
&=&\frac{1}{p}\left\Vert u\right\Vert ^{p}-\lambda 2C_{n}\left\Vert
u(k_{0})\right\Vert _{\infty }^{p}\leq -\frac{1}{p}\left\Vert u\right\Vert
^{p}.
\end{eqnarray*}%
Therefore 
\begin{equation*}
b_{n}=\max_{u\in Y_{n},\left\Vert u\right\Vert =\rho _{n}}J_{\lambda
}(u)\leq 0
\end{equation*}%
and condition $(ii)$ holds.
\end{proof}

\begin{tabular}{l}
Robert Stegli\'{n}ski \\ 
Institute of Mathematics, \\ 
Lodz University of Technology, \\ 
Wolczanska 215, 90-924 Lodz, Poland, \\ 
robert.steglinski@p.lodz.pl%
\end{tabular}

\end{document}